\documentclass[reqno,A4paper]{amsart}
\usepackage{amsmath,amssymb,amsthm,enumerate,mathrsfs,amsxtra,amstext}
\usepackage{eqlist,cite,array,graphicx,color,latexsym,dsfont}
\usepackage[all]{xy}
\usepackage[ansinew]{inputenc}
\newtheorem{thm}{Theorem}[section]

\newtheorem{lem}[thm]{Lemma}

\theoremstyle{definition}
\newtheorem{defn}[thm]{Definition}
\theoremstyle{remark}

\theoremstyle{example}
\newtheorem{exm}[thm]{Example}
\numberwithin{equation}{section}
\begin{document}

\title[Cartan's thorem]{Cartan's theorem for some topological generalized groups}

\author[A. R. Armakan \and M. R. Farhangdoost \and F. Gorlizkhatami \and T. Nasirzadeh]{A. R. Armakan* \and M. R. Farhangdoost** \and F. Gorlizkhatami*** \and T. Nasirzadeh****}
\newcommand{\acr}{\newline\indent}
\address{\llap{*\,}Department of Mathematics,\acr
College of Sciences,\acr
Shiraz University,\acr
P.O. Box 71457-44776,\acr
Shiraz, \acr
 Iran.}
\email{r.armakan@shirazu.ac.ir}
\address{\llap{**\,}(Corresponding author) Department of Mathematics,\acr
College of Sciences,\acr
Shiraz University,\acr
P.O. Box 71457-44776,\acr
Shiraz, \acr
 Iran.}
\email{farhang@shirazu.ac.ir}
\address{\llap{***\,}Department of Mathematics,\acr
College of Sciences,\acr
Shiraz University,\acr
P.O. Box 71457-44776,\acr
Shiraz, \acr
 Iran.}
\email{faranakgolrizkhatami@yahoo.com}
\address{\llap{****\,}Department of Mathematics,\acr
Shahid Bahonar University,\acr
P.O. Box 76169-14111,\acr
Kerman,\acr
 Iran.}
\email{tnasirzade@yahoo.com}
\thanks{This paper is supported by grant no. 92grd1m82582 of Shiraz university, Shiraz, Iran.}
\maketitle
\begin{abstract} In this paper we show that topological subgroupoids of Lie groupoids, under special circumstances are Lie subgroupoids. Giving an example, we indicate that having the same topological dimension is a necessary condition for topological subgroupoids to be Lie subgroupoids. Also, we provide some conditions for double subgroupoids to become double Lie subgroupoids. Moreover, we illustrate that having the same conditions as the Cartan's theorem for Lie groups, helps us prove the same theorem for generalized Lie groups.
\end{abstract}

\section{Introduction}
Lie groups provide a way to express the concept of a continuous family of symmetries for geometric objects. There exist a correspondence between Lie groups as geometric object and Lie algebras as linear objects. By differentiating the Lie group action, you get a Lie algebra action, which is a linearization of the group action. As a linear object, a Lie algebra is often a lot easier to work with than working directly with the corresponding Lie group. Whenever you study different kinds of differential geometry (Riemannian, Kahler, symplectic, etc.), there is always a Lie group and Lie algebra lurking around either explicitly or implicitly.
It is possible to learn each particular specific geometry and work with the specific Lie group and Lie algebra without learning anything about the general theory. However, it can be extremely useful to know the general theory and find common techniques that apply to different types of geometric structures.
Moreover, the general theory of Lie groups and algebras leads to a rich assortment of important explicit examples of geometric objects.
So importance of Lie groups leads to importance of its generalizations. In this paper we deal with three different kinds of Lie groups generalizations, namely, Lie groupoids, Double Lie groupoids and generalized Lie groups or top spaces.

The groupoid was introduced by H. Barant in 1926. C. Ehresmann used the concept of Lie groupoid as an essential tool in topology and differential geometry around 1950.

Double Lie groupoid in double category was interestingly introduced by K. Mackenzie \cite{mac}. A double Lie groupoid is essentially a groupoid object in the category of Lie groupoid. We can present a double Lie groupoid as a square

$$\xymatrix{
V \ar @{>} @<2pt>[d] \ar @{>} @<-2pt>[d] \ar @{<-} @<2pt>[r] \ar @{<-} @<-2pt> [r] &D \ar @{>} @<2pt>[d] \ar @{>} @<-2pt>[d]\\
M \ar @{<-} @<2pt>[r] \ar @{<-} @<-2pt> [r] &H}$$
where each edge is a groupoid and the various groupoid structures satisfy certain compatibility conditions.

Top spaces as a generalization of Lie groups was introduced by M. R. Molaei in 1998. In this generalized field, several authors (Araujo, Molaei, Mehrabi, Oloomi, Tahmoresi, Ebrahimi, etc.) have studied different aspects of generalized groups and top spaces \cite{mo,mb}.



In section 2, introducing some basic definitions, we prove that every subgroupoid of a Lie groupoid with the same dimention, is a Lie subgroupoid and show the analogous for double Lie groupoids.

In section 3 we provide some useful tools to prove a theorem similar to Cartan's in Lie groups case i.e. we show that each closed generalized subgroup of a top space is a top subspace.

\section{Generalized subgroups as Lie generalized subgroups}
E. Cartan showed that every closed subgroup of a Lie group, is a Lie subgroup. In this section we give some conditions under which every subgroupoid (double subgroupoid) of a Lie groupoid (double Lie subgroupoid) is a Lie subgroupoid (double Lie subgroupoid).
 A groupoid is a category in which every arrow is invertible. More precisely, a groupoid consists of two sets $G$ and $G_{0}$ called the set of morphism or arrows and the set of objects of groupoid respectively, together with two maps $\alpha , \beta:G\longrightarrow G_{0}$ called source and target maps respectively, a map $1_{0}:G_{0}\longrightarrow G, \ x\longmapsto x_{0}$ called the object map, an inverse map $i:G\longrightarrow G_{0}, \ a\longmapsto a^{-1}$ and a composition $ G_{2}=G_{\ \alpha}\times_{\beta}G\longrightarrow G, \ (b,a)\longmapsto boa$ defined on the pullback set

$$G_{\ \alpha}\times_{\beta}G=\lbrace(b,a)\in G\times G\vert \alpha(b)=\beta(a)\rbrace.$$
These maps should satisfy the following conditions:
\begin{itemize}

\item[i.]$\alpha(boa)=\alpha(a)$ and $\beta(boa)=\beta(b)$ for all $(boa)\in G_{2}$;

\item[ii.] $co(boa)=(cob)oa$ such that $\alpha(b)=\beta(a)$ and $\alpha(c)=\beta(b)$, for all

$a,b,c\in G$;

\item[iii.]$\alpha(1_{x})=\beta(1_{x})=x$, for all $x\in G_{0}$;

\item[iv.]$ao1_{\alpha(a)}=a$ and $1_{\beta(a)}oa=a$, for all $a\in G$;

\item[v.]$\alpha(a^{-1})=\beta(a)$ and $\beta(a^{-1})=\alpha(a)$, for all $a\in G$.\cite{mac}
\end{itemize}

 Let $(G,G_{0})$ be a groupoid, $M$ be a manifold and $w:M\longrightarrow G_{0}$ be a submersion. An action of $G$ on $M$ via $w$ is a smooth map $\varphi:G_{\ \alpha}\times_{w}M\longrightarrow M, \ (a,x)\longmapsto a.x$, satisfying the conditions:
\begin{itemize}
\item[i.] $w(a.x)=\beta(a)$;

\item[ii.] $b.(a.x)=(boa).x$;

\item[iii.] $1_{w(X)}.x=x$.\cite{fa}
\end{itemize}

\begin{defn}\label{GSub}
 A groupoid $(H,H_{0})$ with $(i,i_{0})$ is called subgroupoid of $(G,G_{0})$, if $$i:H\longrightarrow G,~~~~~~~~~~~~~~~~~~~~~~~~~~~~~~~~~~~~~~~i_{0}:H_{0}\longrightarrow G_{0},$$ are injective and $(i,i_{0})$ is a morphism of groupoids.
\end{defn}
Here is an example of definition \ref{GSub}.
\begin{exm}
Let $G$ be a groupoid acting on $M$. For each $m\in M$, the set
$$H(m)=\lbrace a\in G\vert a.m=m\rbrace,$$
is a subgroupoid of $G$, which is called the stabilizer of $m$ in $G$. Let $G^{'}=H(m)$ and $$G_{0}^{'}=\alpha(G^{'})\bigcap\beta (G^{'}),$$ we show that $(G^{\prime}, G_{0}^{\prime})$ is the subgroupoid of $(G,G_{0})$. See that $\alpha (G^{'})\subset G_{0}^{'}$, $\beta (G^{'})\subset G_{0}^{'}$ and $aob\in G^{'}$, for all $a,b\in G^{'}$. Let $g\in 1_{G_{0}^{'}}$, then $g=1_{x}$, for some $x\in G_{0}^{'}$. Therefore,
$x\in \beta(G^{'})$. In other words there exists $ g^{'}\in G^{'}$ such that $x=\beta(g^{'})$.
Therefore $g^{'}.m=m$ and we have
$$g.m=1_{x}.m=1_{\beta(g^{'})}.m=1_{w(g^{'}.m)}.m =1_{w(m)}.m=m.$$
Let $a\in G^{'},$ then
$$a^{-1}.m=a^{-1}.(a.m)=1_{\alpha(a)}.m=1_{w(m)}.m=m.$$
Hence $(G^{\prime},G_{0}^{\prime})$ with the restriction of source and target maps of $(G,G_{0})$ is a groupoid, $i$ and $i_{0}$ are injective and $(i,i_{0})$ is a morphism of groupoids.
\end{exm}

Now we recall the Lie groupoids and morphism of groupoids. A groupoid $(G,G_{0})$ is called Lie groupoid if $G$ and $G_{0}$ are manifolds, $\alpha$ and $\beta$ are surjective submersions and the composition is a smooth map. For example, any manifold $M$ may be regarded as a Lie groupoid on itself with $$\alpha=\beta=id_{M},$$ and every element a unity.\cite{mac}

 Let $(G,G_{0})$ and $(G^{'},G_{0}^{'})$ be groupoids. A morphism $G\rightarrow G^{\prime}$ is a pair of maps $F:G\rightarrow G^{\prime}$, $f:G_{0}\rightarrow G_{0}^{\prime}$ such that $\alpha^{\prime} oF=fo\alpha$, $\beta^{\prime} oF=fo\beta$ and $F(hg)=F(h)F(g)$. Let $(G,G_{0})$ and $(G^{'},G_{0}^{'})$ be Lie groupoids, then $(F,f)$ is a morphism of Lie groupoids if $F$ and $f$ are smooth.\cite{mz}

\begin{exm}
 For any groupoid  $(G,G_{0})$, the map
 $$\chi=(\beta,\alpha):G\rightarrow G_{0}\times G_{0}, g\mapsto(\beta(g),\alpha(g)),$$
 is a morphism from $G$ to $G_{0}\times G_{0}$.
\end{exm}
\begin{defn}
\cite{mac}
 Let $(G,G_{0})$ be a Lie groupoid. A Lie subgroupoid of $(G,G_{0})$ is a Lie groupoid $(H,H_{0})$ together with injective immersions $i:H\rightarrow G$ and $i_{0}:H_{0}\rightarrow G_{0}$ such that $(i,i_{0})$ is a morphism of Lie groupoid.
\end{defn}

If $(i,i_{0})$ is a morphism of Lie groupoids and $i$ is an injective immersion, then $i_{0}$ is an injective immersion too.

\begin{defn}
\cite{man} Let $G$ be a manifold. A Lie algebroid on $G$ is a vector bundle
$( E, p,G)$ together with a vector bundle map $\sigma:E\longrightarrow TG$ called the anchor of $E$, and a bracket $[.,.]$ on sections of $E$, i.e. $\Gamma(E)$, which is R-bilinear and alternating, satisfying the Jacobi identity, and is such that

$$[X,fY]=\sigma(X).fY+f[X,Y],$$
$$\sigma([X,Y])=[\sigma(X),\sigma(Y)],$$
where $X,Y\in \Gamma(E), f\in c^{\infty}(M)$.
\end{defn}

 Let $(G, G_{0})$ be a groupoid with the object map $1_{0}:G_{0}\rightarrow G$ and let
$$T^{\alpha}G:=ker(T\alpha),$$
where $T\alpha$ be the bundle map introduced by $\alpha: G\rightarrow G_{0}$, i.e.
$$T\alpha:TG\longrightarrow TG_{0}.$$
The pullback bundle $AG:=1_{0}^{\ast}(T^{\alpha}G_{0})$ is a Lie algebroid of the Lie groupoid $(G,G_{0})$. We have $AG=\cup_{x\in G_{0}}T_{1_{x}}(G_{x})$, where $G_{x}=\alpha^{-1}(x)$.

Now we recall the topological dimension.
 A collection $\mathcal{A}$ of subsets of the space $X$ is said to have order $m+1$ if some point of $X$ lies in $m+1$ elements of $\mathcal{A}$ and no point of $X$ lies in more than $m+1$ elements of $\mathcal{A}$. We recall that given a collection $\mathcal{A}$ of subsets of $X$, a collection $\mathcal{B}$ is said to refine $\mathcal{A}$, or to be a refinement of $\mathcal{A}$, if each element $\beta$ of $\mathcal{B}$ is contained in at least one element of $\mathcal{A}$. A space $X$ is said to be finite dimensional if there is some integer $m$ such that for every open covering $\mathcal{A}$ of $X$, there is an open covering $\mathcal{B}$  of $X$ that refines $\mathcal{A}$ and has order at most $m+1$. The topological dimension of $X$ is defined to be the smallest value of $m$ for which this statement holds. The topological dimension of any m-manifold is at most $m$.\cite{man}

In the following theorem we present an important fact about Lie subgroupoids.

\begin{thm}\label{asli}
 Let $(H,H_{0})$ with $(i,i_{0})$ be a subgroupoid of a Lie groupoid $(G,G_{0})$ which satisfies the following conditions. $H_{0}$ is a submanifold of $G_{0}$ and $i_{0}$ is a submersion map, $i$ is injective  and the topological dimension of $i(H)$ and the dimension of  $G$ are equal. Then $(H,H_{0})$ is a Lie subgroupoid.
\end{thm}

\begin{proof}
Step 1.

We show that $i(H)$ is a submanifold of $G$:  Suppose $E$ be a algebroid of $(G,G_{0})$, since $E$ is a vector bundle over $G$, there is a smooth, surjective and locally trivial map $P:E\longrightarrow G$, i.e. for every $x\in i(H)$ there exists a neighborhood $U\subset G$ of $x$ and a fiber-preserving diffeomorphism
$$\varphi: P^{-1}(U)\longrightarrow U\times R^{k}.$$
Therefore $\varphi$ maps $P^{-1}(i(H)\cap U)$ to $(i(H)\cap U)\times R^{k}$.
The restriction of $\varphi^{-1}$ to $U$ is a diffeomorphism onto $\varphi^{-1}\vert_{U}(U)$.
Suppose $\psi:W\rightarrow R^{n}$ is a chart of $E$ such that $ P^{-1}(x)\subset W$. Let $$V=\varphi(W\cap\varphi^{-1}\vert_{U}(U))\cap i(H),$$
 then $\psi o\varphi^{-1}\vert_{V}$ is a chart for $i(H)$.
Since the topological dimension of $H$ and the dimension of $G$ are equal and the topological dimention is topological invariant \cite{pe}, these charts are $C^{\infty}$-related.

Step 2.

 One can see that $H$ is a manifold. We define $U\subset H$ to be open if $i(U)$ is open in $G$.
 For every $h\in H$ there exists $x\in i(H)$ such that $h=i^{-1}(x)$. Suppose $\eta$ is a chart for $x$, then $\eta oi$ is a chart for $h$.

Step 3.

 The groupoid $(H,H_{0})$ with $(i, i_{0})$ is a Lie subgroupoid of $(G,G_{0})$; the map $i:H\rightarrow i(H)$ is a diffeomorphism and so is an immersion. Thus $(i,i_{0})$ is a morphism of Lie groupoids. We have $\alpha_{G}oi=i_{0}o\alpha_{H}$ and $\beta_{G}oi=i_{0}o\beta_{H}$, so $\alpha_{H}$ and $\beta_{H}$ are submersions.
Therefore $ (H,H_{0}) $ is a Lie subgroupoid of $(G,G_{0})$
\end{proof}

 Note that if the topological dimensions of $G$ and $i(H)$ are not the same, the result of the previous theorem is not true.

\begin{exm}
 Consider the groupoids $(R^{2},R)$ and $(H,R)$, where $$H=\{(x,y)\in R^{2}\vert x^{2}=y^{2}\}.$$
 $H$ is a  subgroupoid of $(R^{2},R)$ but $H$ doesn't have the structure of a manifold.
\end{exm}

A double groupoid $(D ,V ,H ,M )$ is a higher-dimensional groupoid involving a relationship for both horizontal and vertical groupoid structures. In following we will define double Lie groupoids in detail.
\begin{defn}\cite{meh}
Let $ \alpha_{H},\beta_{H}:D\rightarrow H $ and $ \alpha_{V},\beta_{V}:D\rightarrow V $ denote source and target maps, respectively and we use $ \alpha $, $ \beta $ to denote source and target maps from $ H $ or $ V $ to $ T $,
the square

$$\xymatrix{
V \ar @{>} @<2pt>[d] \ar @{>} @<-2pt>[d] \ar @{<-} @<2pt>[r] \ar @{<-} @<-2pt> [r] &D \ar @{>} @<2pt>[d] \ar @{>} @<-2pt>[d]\\
M \ar @{<-} @<2pt>[r] \ar @{<-} @<-2pt> [r] &H}$$

is a double Lie groupoid if the following conditions hold:
\begin{itemize}
\item[i.] The horizontal and vertical source and target maps commute;
\begin{center}
$ \alpha o\alpha_{H}=\alpha o\alpha_{V}\quad \quad \beta o\alpha_{H}=\alpha o\beta_{V} $\\
$ \beta o\beta_{H}=\beta o\beta_{V}\quad \quad \alpha o\beta_{H}=\beta o\alpha_{V} $
\end{center}
\item[ii.] $ \alpha_{V}(s_{1} ._{H} s_{2})=(\alpha_{V}s_{1}).(\alpha_{H}s_{2})
, \alpha_{H}(s_{1} ._{V} s_{3})=(\alpha_{H}s_{1}).(\alpha_{H}s_{3}) $ and similar equations hold for $ \beta_{H},\beta_{V}; $
\item[iii.] $ \alpha_{V}(s_{i1})=\beta_{V}(s_{i2}) $ and $ \alpha_{H}(s_{1i})=\beta_{H}(s_{2i}) $, for $ i=1,2,... ;$
\item[iv.] The double source map $ (\alpha_{V},\alpha_{H}):D\rightarrow V_{s\times s}H $ is a submersion.
\end{itemize}
\end{defn}

\begin{defn}\label{LiDoSub}
Let $(D,V,H,M)$ be a double groupoid. A double subgroupoid is a double groupoid $ (D_{1},V_{1} ,H_{1},M_{1} ) $ such that groupoids $ (D_{1} , H_{1}  )$, $ (D_{1} , V_{1} ) $, $( V_{1} , M_{1} ) $ and $( H_{1} , M_{1} ) $ are subgroupoids of $ (D  , H  )$, $ (D  , V ) $, $ (V  , M ) $ and $ (H , M ) $, respectively.
\\
Let $(D,V,H,M)$ be a double Lie groupoid. A double Lie subgroupoid is a double Lie groupoid $ (D_{1},V_{1},H_{1},M_{1})$ together with maps $(j,j_{0},i,i_{0})$ such that Lie groupoids $(D_{1},H_{1}) $, $(D_{1},V_{1})$, $(V_{1},M_{1})$ and $(H_{1},M_{1})$ are Lie subgroupoids of $(D,H)$, $(D,V)$, $(V,M)$ and $(H,M)$, with maps $(j,i)$, $(j,j_{0})$, $(j_{0},i_{0})$ and $(i,i_{0})$, respectively.
\end{defn}
The next theorem is a generalization of theorem \ref{asli} for double Lie groupoids.
\begin{thm}
Let $ (D_{1},V_{1},H_{1},M_{1}) $ with $(j, j_{0}, i, i_{0}) $ be a double subgroupoid of $(D,V,H,M)$. If $dim H_{1} =dim  H$, $dim  D_{1} =dim D $ and $dim V_{1}=dim  V $ and
$ j_{0} $ and $ i_{0} $ are injective and immersion maps, Then $ (D_{1},V_{1},H_{1},M_{1}) $ is a double Lie subgroupoid.

\end{thm}
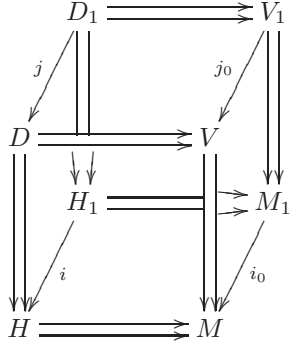
\begin{displaymath}
\xymatrix@!0{
& D_{1} \ar@<-0.5ex>[rrr]\ar@<0.5ex>[rrr]\ar@<-0.5ex>' [dd] [ddd]\ar@<0.5ex>' [dd] [ddd]
& & & V_{1} \ar@<-0.5ex>[ddd]\ar@<0.5ex>[ddd]
\\
& & & &
\\
D \ar@{<-}[uur]^{j}\ar@<-0.5ex>[rrr]\ar@<0.5ex>[rrr]\ar@<-0.5ex>[ddd]\ar@<0.5ex>[ddd]
& & & V \ar@{<-}[uur]^{j_{0}}\ar@<-0.5ex>[ddd]\ar@<0.5ex>[ddd]
\\
& H_{1}\ar@<-0.5ex>' [rr] [rrr]\ar@<0.5ex>' [rr] [rrr]
& & & M_{1}
\\
& & & &
\\
H \ar@{<-}[uur]_{i} \ar@<-0.5ex>[rrr]\ar@<0.5ex>[rrr]
& & & M \ar@{<-}[uur]_{i_{0}}
}
\end{displaymath}
\begin{proof}
We suppose $ (D,V,H,M)$ is a double Lie groupoid and $ (D_{1},V_{1},H_{1},M_{1}) $ is a double subgroupoid. We  show that $ (D_{1},V_{1},H_{1},M_{1}) $ has structure of double Lie groupoid.

Using the proof of theorem \ref{asli} for both vertical sides and both horizontal sides of double Lie groupiod $ (D,V,H,M)$,  we get $ (D_{1},H_{1})$, $(D_{1},V_{1})$, $(V_{1},M_{1})$ and $(H_{1},M_{1})$ are Lie subgroupoid of $ (D,H)$, $(D,V)$, $(V,M)$ and $(H,M)$, respectively. Then by definition \ref{LiDoSub}, $ (D_{1},V_{1},H_{1},M_{1}) $ is a double Lie subgroupoid.
\end{proof}

\section{Cartan's theorem for generalized Lie groups}
Another generalization of Lie groups is called generalized Lie groups or top spaces which arises from the definition of a generalized group. In this section we prove a theorem analogous to Cartan's theorem in Lie groups case.
\begin{defn}
\cite{mo}
A generalized group is a non-empty set $\mathcal{G}$ admitting an operation called multiplication which satisfies the following conditions:
\begin{itemize}
\item[i.] $(g_1 . g_2) . g_3 = g_1 . (g_2 . g_3)$, for all $g_1, g_2, g_3 \in\mathcal{G}$.
\item[ii.] For each $g \in \mathcal{G}$ there exists a unique $e(g)$ in $\mathcal{G}$ such that $$g . e(g) = e(g) . g = g.$$
\item[iii.] For each $g \in \mathcal{G}$ there exists $h \in \mathcal{G}$ such that $g . h = h . g = e(g)$.
\end{itemize}

\end{defn}
In this paper by $e(\mathcal{G})$ we mean $$\{e(g):g\in \mathcal{G}\}.$$  For any generalized group $\mathcal{G}$, and any $g\in \mathcal{G}$,
 $$e^{-1}(e(g))=\{h\in \mathcal{G}|e(h)=e(g)\},$$ has a canonical group structure.
 If $e(g)e(h)=e(gh)$ for all $g,h\in \mathcal{G}$ then $e(\mathcal{G})$ is an idempotent semigroup with this product.

A top space is a smooth manifold which its points can be (smoothly) multiplied together by a generalized group operation and generally its identity is a semigroup morphism, i.e.
\begin{defn}\cite{bull}
A top space $T$ is a Hausdorff d-dimensional differentiable manifold which is
endowed with a generalized group structure such that the generalized group operations:
\begin{itemize}
\item[i.] $. : T \times T\rightarrow T$ by $(t_1, t_2) \mapsto t_1 . t_2$ which is called the multiplication map;
\item[ii.] $^{-1} : T \rightarrow T$ by $t \mapsto t^{-1}$ which is called the inverse map;\\
are differentiable and it holds
\item[iii.] $e(t_1 . t_2) = e(t_1) . e(t_2)$, for all $t_1
, t_2 \in T$ .
\end{itemize}
\end{defn}
Throughout this paper by $T_{a}$ we mean $T\cap e^{-1}(e(a))$.
\begin{defn}\cite{bull}
If $T$ and $S$ are two top spaces, then a homomorphism $f : T \rightarrow S$
is called a morphism if it is also a $C^{\infty}$ map.
If $f$ is a morphism, by $f_{a}$ we mean $f|_{e^{-1}(e(a))}$.
\end{defn}
 For a smooth manifold $M$, the set of all smooth functions from $M$ to $M$ such that their restriction to a submanifold of $M$ is a diffeomorphism i.e. partial diffeomorphisms of $M$ is denoted by $D_P(M)$.
If $T$ is a top space, we call an immersed submanifold $S$ of $T$ a top subspace, if it is a top space.
Before we continue, we have to recall some tools. The following definitions, examples and theorems up to theorem \ref{form} is from \cite{ma}.
\begin{defn}\label{action}
An action of a top space $T$ on a smooth manifold M is a
map $$\phi:T\rightarrow D_P(M),$$ which satisfies the following conditions:
\begin{itemize}
\item[i.] For every $i\in e(T)$ the map $e^{-1}(i)\times M\rightarrow M$ which maps $(t,m)$ to $\phi_t(m)$ is a smooth function;
\item[ii.] $\phi_{ts}=\phi_t\circ\phi_s$, for all $t,s\in T$.
\end{itemize}
\end{defn}
Now by using the following theorem, we give an example of the definition above.

\begin{thm}
Let $T$ be a generalized group such that $e(t)e(s)=e(ts)$, for any $t,s\in T$. Then there is a generalized group isomorphism between $T$ and $$e(T)\ltimes\{G_i\}_{i\in e(T)},$$where $G_i=e^{-1}(i)$, for all $i\in e(T)$ and $$e(T)\ltimes\{G_i\}_{i\in e(T)}=\{(i,g)|g\in G_i\},$$ by the production rule
$$(i_{1},g_{1})\ltimes(i_{2},g_{2})=(i_{1}i_{2},g_{1}g_{2}),\quad i_{1},i_{2}\in e(T)
, g_{1}\in G_{i_{1}}, g_{2}\in G_{i_{2}}.$$
\end{thm}

\begin{exm}
For a top space $T$, $Ad:T\rightarrow D_P(T),$ where $Ad_t(s)=tst^{-1}$ is an action of top spaces on manifolds which is called adjoint action. Observe that $ad_{t}:e^{-1}(t)\rightarrow e^{-1}(t)$ is a diffeomorphism.
\end{exm}

\begin{defn}
If $G$ is a Lie group and $M$, a smooth manifold, a partial action of $G$ on $M$ is a map
$$\varphi:G\rightarrow D_P(M),$$ such that the map $G\times M\rightarrow M$ is a smooth function and $\varphi^{gh}=\varphi^{g}\circ\varphi^{h}$, for all $g,h\in G$.
\end{defn}

\begin{thm}\label{th;act1}
Let $T$ be a top space. Then $\phi$ is an action of $T$ on $M$ if and only if there exists a family of partial actions $\{\varphi_i\}_{i\in e(T)}$ of $e^{-1}(i)$ on $M$, for any $i\in e(T)$, such that $\varphi_{e(ts)}^{ts}=\varphi_{e(t)}^{t}\circ\varphi_{e(s)}^{s}$ and $\phi_t=\varphi_{e(t)}^{t}$, for all $t,s\in T$.
\end{thm}

\begin{thm}\label{form}
Let $T$ be a top space and $T_ie^{-1}(i)$ the tangent space of $e^{-1}(i)$ at $i\in e(T)$. Then the vector space of all partial left invariant vector fields on $T$ is isomorphic to $$\bigoplus_{i\in e(T)}T_ie^{-1}(i).$$
\end{thm}

In order to reach our aim, we must show that there exists a map from Lie algebra of a top space to the top space itself, which acts like the exponential map in Lie group case.
Consider the exponential map for Lie groups denoted by $exp$.
\begin{lem}
Let $T$ be a top space. Then $exp_{a}$ is a local diffeomorphism for every $a\in e(T)$.
\end{lem}
\begin{proof}
The action $\alpha$ of a top space on a manifold is a partial diffeomorphism for each $t\in T$. So $\alpha_{a}$ is a diffeomorphism for every $ a\in e(T)$. So the $exp_{a}=\gamma_{\alpha_{a}}(1)$ is a local diffeomorphism.
\end{proof}
Now we are ready to prove the following important theorem.
\begin{thm}
Each closed generalized subgroup of a top space is a top subspace.
\end{thm}
\begin{proof}
Let $T$ be a top space and $S$, a closed generalized subgroup of it. We know that $T_{a}$ is a Lie group for every $a \in e(T)$. So $S_{a}$ is a closed subgroup, since a Lie subgroup of $T_{a}$, for every $a \in e(S)$. Denote by $\mathfrak{t}$, the Lie algebras of $T$. Note that by theorem \ref{form} the vector space of all partial left invariant vector fields on $t$ is isomorphic to $\bigoplus_{a\in e(t)}\mathfrak{t_{a}}$, where $\mathfrak{t_{a}}$ is the Lie algebra of $T_{a}$. Now consider the set
 $$\mathfrak{s}=\{ X\in \mathfrak{t}| exp(tX)\in S, \forall t\in \mathds{R} \},$$
 where the exponential map $exp=\bigoplus_{a \in e(T)}exp_{a}$ is a local diffeomorphism. Obviously $\mathfrak{s}$ is a linear subspace of $\mathfrak{t}$ since every $\mathfrak{s}_{a}$ is a linear subspace of $\mathfrak{t}_{a}$, for all $a$ in $e(S)$. Open sets of $\mathfrak{s}$ are unions of open sets on each $\mathfrak{s_{a}}$. We found a local diffeomorphism between $\mathfrak{s}$ and $S$; therefore, by use of the left transitions we can find an open neighbourhood in $\mathfrak{s}$, for all open sets of $S$. Thus, we can find a smooth chart for each open set of $S$ since $s$ is a linear space. Hence, $S$ is a smooth embedded submanifold which inherits the subspace topology.

\end{proof}
The question remains to be asked is under which conditions different from the ones given in this paper, we can deduce the same results.
\section{Acknowledgments}
This paper is supported by grant no. 92grd1m82582 of Shiraz university, Shiraz, Iran.

\bibliographystyle{<bibstyle>}
\bibliography{<bib>}

\end{document}